\documentclass[12pt]{amsart}

\usepackage{fourier}
\usepackage{calrsfs}

\usepackage{amssymb}
\usepackage{amsfonts}

\usepackage{color}
\usepackage{epsfig}
\usepackage{graphicx}
\usepackage{graphics}
\usepackage{amssymb}
 \usepackage{fullpage}
\usepackage{comment}
\usepackage{hyperref}

\newfont{\vlte}{eufm10 at 22pt}
\newfont{\lte}{eufm10 at 18pt}
\newfont{\smb}{msbm6}
\newfont{\mmb}{msbm8}
\newfont{\tmb}{msbm10}
\newfont{\lmb}{msbm10 at 18pt}

\newcommand{\ord}{\mbox{ord}}

\newcounter{hw}
\newcommand{\be}{\begin{enumerate}}
\newcommand{\ee}{\end{enumerate}}

\chardef\secsym=129
\newcommand{\calA}{{\mathcal A}}

\newcommand{\calI}{{\mathcal I}}

\newcommand{\calP}{{\mathcal P}}
\newcommand{\calQ}{{\mathcal Q}}
\newcommand{\calR}{{\mathcal R}}
\newcommand{\calS}{{\mathcal S}}
\newcommand{\calT}{{\mathcal T}}

\newcommand{\calZ}{{\mathcal Z}}

\newcommand{\C}{{\mathbb C}}
\newcommand{\F}{{\mathbb F}}

\newcommand{\Q}{{\mathbb Q}}
\newcommand{\R}{{\mathbb R}}
\newcommand{\Z}{{\mathbb Z}}

\newcommand{\pp}{{\mathfrak p}}
\newcommand{\Pp}{{\mathfrak P}}
\newcommand{\DD}{{\mathfrak D}}
\newcommand{\Aa}{{\mathfrak A}}
\newcommand{\Tt}{{\mathfrak T}}
\newcommand{\qq}{{\mathfrak q}}

\newcommand{\ttt}{{\mathfrak t}}

\chardef\sha=88
\newtheorem{theorem}{Theorem}[section]
\newtheorem{lemma}[theorem]{Lemma}
\newtheorem{corollary}[theorem]{Corollary}

\theoremstyle{definition}
\newtheorem{definition}[theorem]{Definition}

\newtheorem{example}[theorem]{Example}

\newtheorem{construction}[theorem]{Construction}

\theoremstyle{remark}
\newtheorem{remark}[theorem]{Remark}

\newtheorem{notationassumption}[theorem]{Notation and Assumptions}

  \theoremstyle{plain}
\newtheorem{introprop}{Proposition}
\newtheorem{introthm}[introprop]{Theorem}

\begin{document}
\bibliographystyle{plain}%
 \title{Definability and Decidability in Infinite Algebraic Extensions }%
\author{Alexandra Shlapentokh and Carlos Videla}%
\thanks{The first author has been partially supported by NSF grant DMS-0650927 and by a grant from John Templeton Foundation.}
\address{Department of Mathematics \\ East Carolina University \\ Greenville, NC 27858, USA}%
\email{shlapentokha@ecu.edu }
\urladdr{www.personal.ecu.edu/shlapentokha} 
\address{Department of Mathematics, Physics and Engineering, Mount Royal University, Calgary, Alberta, Canada T3E 6K6}
\email{cvidela@mtroyal.ca}
\begin{abstract}
We use a generalization of a construction by Ziegler to show that for any field $F$ and any countable collection of countable subsets $A_i \subseteq F, i \in \calI \subset \Z_{>0}$  there exist infinitely many fields $K$ of arbitrary positive transcendence degree over $F$ and of infinite algebraic degree  such that each $A_i$ is first-order definable over $K$.  We also use the construction to show that many infinitely axiomatizable theories of fields which are not compatible with the theory of algebraically closed fields are finitely hereditarily undecidable.
\end{abstract}
\date{\today}
\maketitle
\section{Introduction}
In this paper we investigate an interesting definability phenomenon occurring in some extensions of positive transcendence and infinite algebraic degree and derive a number of model theoretic consequences.   While we know a great deal (though far from everything) about first-order definability over number fields and function fields, especially function fields over a finite field of constants, and over fields which are ``close'' to  being algebraically closed, we know substantively less about infinite algebraic extensions of rational fields which are ``far'' from algebraic closure.  (See for example \cite{Da2}, \cite{Dencollection}, \cite{Eispadic}, \cite{JR}, \cite{Jarden2}, \cite{Mate}, \cite{MB3}, \cite{PO4}. \cite{PP}, \cite{Rob1}, \cite{Rob2}, \cite{Rum}, \cite{Rum1}, \cite{Sh34}, \cite{Dries4},\cite{Dries3}, \cite{V1}, \cite{V2}, \cite{V3}   for a description of first-order and existential definability results over number fields and function fields.  This list is very far from being exhaustive and is just supposed to give the reader a sample of the results on the matter.)

The questions of definability are usually considered from the following point of view: given a field or a ring, describe the definable sets.  In this paper we consider a different approach: given a subset of a field (or a countable collection of subsets), describe  field extensions where this subset (or each set in the collection) is definable.  Our construction is a generalization of a construction by Martin Ziegler (see \cite{Ziegler}) which he used to show, among other things,  that $\Z$ is definable in a class of fields.  One of the main results of this paper can be stated as the following theorem.
\begin{introthm}%
Let $F$ be any field.  Then for any countable collection of countable subsets $$A_i \subseteq F, i\in \calI \subset \Z_{>0}$$  there exist infinitely many fields $K$ of arbitrary non-zero transcendence degree  over $F$ and of infinite algebraic degree  such that  each $A_i$ is first-order definable over $K$.
\end{introthm}%

While Ziegler's paper had interesting definability results, its primary focus was proving the undecidability of finitely axiomatizable subtheories of various theories.  The main idea which enabled Ziegler to prove the undecidability results can be summarized in the following argument. If $M$ is one of the fields: $L_{p}$ (algebraic closure of a $\F_p(t)$, where $\F_p$ is a finite field of $p$ elements, and $t$ is transcendental over $\F_p$), $\C$, $\R$, or $\Q_p$, then for every rational prime $q$, not equal to the characteristic of the field, there exists a field $K_q \subset M$ with the following properties:
\be
\item if $H$ is a subfield of $M$,  $K_q \subseteq H$, and $[H:K_q] < \infty$, then either $$[H:K_q]=1$$ or $$[H:K_q]\equiv 0 \mod q;$$
\item Each field $K_q$ is strongly undecidable, i.e. any theory for which $K_q$ is a model is undecidable.
\ee
Let $\Omega$ be a non-principal ultra-filter on the set of rational primes, let $$K=\prod_q K_q/\Omega,$$ and let $$\hat M=\prod_qM/\Omega.$$  If $\mathfrak T$ is a finitely axiomatizable theory for which $M$ is a model, then  $\hat M$ and $K$ are also models of $\mathfrak T$, and therefore by {\L}o\u{s}'s Theorem for at least one $q$ we have that  $K_q$ is a model of  $\mathfrak T$, implying that  $\mathfrak T$ is undecidable.
In his paper Ziegler considered among others finitely axiomatizable fragments of the following theories:
\begin{itemize}
\item $\Tt^{A}_{p,q}$ -- the theory of a field of characteristic $p \geq 0$, where every irreducible polynomial is either of degree 1 or of degree divisible by $q$.  (If $p=0$, this theory is compatible with the theory of $\C$, and if $p >0$, then this theory is compatible with the theory of $L_p$.)
\item $\Tt^{R}_2$ -- the theory of a formally real field where all irreducible polynomials have degree 1 or 2.
\item $\Tt^R_q, q \not=2$ -- the theory of a formally real field where the degree of any irreducible polynomial is either 1 or is divisible by $q$, and the field is dense in its real algebraic closure.
 \item $\Tt^H_q, q \not=2$ -- the theory of a formally p-adic field where the degree of any irreducible polynomial that has a zero in the $p$-adic closure is either 1 or is divisible by $q$, and the field is dense in its $p$-adic algebraic closure.
\end{itemize}
\begin{remark}
In general, {\it for any} $M$, given the construction above, if $\mathfrak T$ is a collection of statements that some polynomials have a root in a field while others are irreducible and if $M$ is a model of $\mathfrak T$,  then, as before, $\hat M$ is a model of $\mathfrak T$ by {\L}o\u{s}'s Theorem, $K$ is algebraically closed in $\hat M$ by construction, and because $K$ is algebraically closed in $\hat M$, all the statements of $\mathfrak T$  will be true in $K$.
\end{remark}
It follows form the discussion above that every finite subtheory of the theories listed above is undecidable and therefore all these theories are finitely hereditarily undecidable.  We should also note here that in the case of  positive characteristic the undecidability of each  $K_q$ is obtained by interpreting the theory of graphs in these fields, and in the case of zero  characteristic, $\Z$ is defined over each $K_q$.

Using a generalized version of Ziegler's construction we show the following.

\begin{introthm}
\label{introthm2}
Let $U$ be a countable field, let  $\calQ$ be an infinite set of rational primes not including the prime equal to the characteristic of the field,  let $\calR=\{R_q(T): q \in \calQ\} \subset U(T)$, where all the zeros and poles of $R_q$ are in $U$,  at least one pole or zero of $R_q(T)$ is of degree not divisible by $q$, and for any $x \in U$ there exists $y \in U$ such that $y^q=R_q(x)$.     Let  $\calP=\{P_i(T), i \in \Z_{>0}\} \subset U[T]$ be such that  for any $i \in \Z_{>0}$ we have that $P_i(T)$ is irreducible over $U$ and $P_i$ does not factor in any extension of degree $q \in \calQ$.  Let $\calZ=\{Z_i(T), i \in \Z_{>0}\} \subset U[T]$ be such that  for any $i \in \Z_{>0}$ it is the case that $Z_i(T)$ has a root in $U$. Now if  $\Tt_{U, \calR,\calP,\calZ}$ is a first-order theory of fields in the language with a constant symbol for every element of $U$, consisting of the atomic diagram of $U$ and the following statements:
\be
\item \label{it:2} $\forall q \in \calQ:  (\forall x  \exists y: y^q=R_q(x)$);
\item \label{it:3} $\forall i: (P_i(T)$ is irreducible);
\item \label{it:4} $\forall i: (Z_i(T)$ has a root),
\ee

then any finite subtheory  of $\Tt_{U, \calR,\calP, \calZ}$ is undecidable.
\end{introthm}

\begin{remark}
In Part \ref{it:3}, the polynomials are listed explicitly.  In  Part \ref{it:4}, the polynomials can be listed explicitly or we may have statements asserting that all polynomials of a certain form (e.g. degree) have a root.
\end{remark}

To prove the theorem above we follow Ziegler's construction except that in the case of positive characteristic we define a polynomial ring inside our fields $K_q$.  In order to do this we need a proposition below.
\begin{introprop}
\label{introprop}
     There exists $G \models \Tt_{U, \calR,\calP, \calZ}$  of any transcendence degree over $U$.
\end{introprop}
\begin{proof}
Let $G_0=U$.  We show how to construct a field $G_1$ of transcendence degree one over $G_0$ satisfying the same conditions.  Let $H_{0}=G_0(t)$,  and let $H$ be the smallest field in the algebraic closure of $H_0$ containing $\sqrt[q]{R_q(x)}$ for every $q \in \calQ$ and every $x \in G$.  In this case every finite extension of $G$ contained in $H$ is of degree $\prod_{i=1}^m q_i$, with each $q_i \in \calQ$ and can be decomposed into a sequence of extensions each of degree $q_i \in \calQ$ .  By assumption on $\calP$, all the polynomials in $\calP$  will remain irreducible under this extension.
\end{proof}


\section{Technical preliminaries}
\setcounter{equation}{0}
In this section we discuss several properties of function fields to be used in our construction.  We first explain what we mean by a function field.
\begin{definition}[Function Fields]
 For a field $C$ and an element $t$  transcendental over $C$, we say that a field $G$ is a one-variable function field over a field of constants $C$ if $G/C(t)$ is a finite extension and $C$ is algebraically closed in $G$.
\end{definition}
Below we also use primes of function fields to prove that the fields we are constructing have the required attributes.   For a general introduction to primes of function fields and their properties, the reader is referred to \cite{C}.  We define  primes of a function field below.
\begin{definition}[Primes of Function Fields]%
Let $G$ be a one-variable function field over a field of constants $C$.  Let $t\in G\setminus C$.  Let $O_G$ be the integral closure of $C[t]$ in $G$ and let $O_{\infty}$ be the integral closure of $C[\frac{1}{t}]$  in $G$.  We define a prime of $G$ to be either a prime ideal $\pp$ of $O_G$ or a prime ideal  $\qq$ of $O_{\infty}$ such that $\qq \cap C[\frac{1}{t}]$ is the prime ideal generated by $\frac{1}{t}$.  The primes of $O_{\infty}$ will be called infinite primes.
\end{definition}%
We now define order at a prime.
\begin{definition}[Order at a prime]
Let $C, t, G, O_G, O_{\infty}$ be as above.  Let $\pp$ be a prime of $G$ which is a prime ideal of $O_G$.   For $x \in O_G$ if  $m \in \Z_{\geq 0}$ is such that $x \in \pp^{m+1} \setminus \pp^m$, then we say that $\ord_{\pp}x=m$.  If $w\in K, w \not = 0$ and $w=\frac{x}{y}$, where $x,y \in O_G$, then define $\ord_{\pp}w=\ord_{\pp}x-\ord_{\pp}y$.  Define $\ord_{\pp}0 =\infty$.  The order at  infinite primes is defined in an analogous manner.
\end{definition}
\begin{remark}
\label{rem:minorder}
Given our definition of order, it is easy to see that for all $a, b \in G$ and all $G$-primes $\pp$ it is the case that $\ord_{\pp}(a + b) \geq \min(\ord_{\pp}a, \ord_{\pp} b)$, and if $\ord_{\pp}a < \ord_{\pp}b$ then $\ord_{\pp}(a + b) =\ord_{\pp}a$
\end{remark}
\begin{remark}
\label{rem:const}
Observe that for any non-zero constant element $a$ of $G$ and any prime $\pp$ of $G$ it is the case that $\ord_{\pp}a=0$.  At the same time if $z \in G \setminus C$, then for at least one prime $\pp$ of $G$ we have that $\ord_{\pp}z >0$ ($\pp$ is called a zero of $z$), and for at least one prime $\qq$ of $G$ we have that $\ord_{\qq}z <0$ ($\qq$ is called a pole of $z$).  See \cite{C}, Chapter I, Section 3, Corollary 3 for more details.
\end{remark}
We need to define one more object to facilitate the discussion  below:
\begin{definition}[Divisors ]
\label{def:divisor}
Let $C, t, G$ be as above.  Let $\calP(G)$ be the set of all the primes of $G$.  Any finite product $$\Aa=\prod_{\pp \in \calP(G)} \pp^{a(\pp)}, a(\pp) \in \Z$$ is called a divisor of $G$.  If $x \in G$, then a formal product $$\DD(x) = \prod_{\pp \in \calP(G)}\pp^{\ord_{\pp}x}$$ is called the divisor of $x$. (It can be shown that for any $x \in G$ the divisor of $x$ has only finitely many terms with non-zero exponents.)  The set of all divisors form an abelian group under multiplication.
\end{definition}

Next we need to define ramification, degree and relative degree.
\begin{definition}[Ramification]
Let $G$ be as above and let $H$ be a finite extension of $G$.   If $\pp_G$ is a prime ideal of $O_G$, then $\pp_G O_H$ is not necessarily a prime ideal and it can be written uniquely as a product of prime ideals of $O_H$:
\[%
\Pp = \prod_{i=1}^m\pp_{H,i}^{e_i}
\]%
The power $e(\pp_{H,i}/\pp_G)=e_i$ is called the ramification degree of $\pp_{H,i}$ over $\pp_G$ or over $G$.  We also say that $\pp_{H,i}$ lies above $\pp_G$ in $H$ and $\pp_G$ lies below $\pp_{H,i}$ in $G$.  Ramification degree for infinite primes is defined analogously.
\end{definition}
\begin{remark}
\label{rem:multe}
In the notation above, if $x \in G$, then $\ord_{\pp_G}x = e_i\ord_{\pp_{H,i}}x$.
\end{remark}

\begin{definition}[Degree and Relative Degree]
In the above notation, if $\pp_G$ is a prime of $O_G$, then $\pp_G$ is a maximal ideal of $O_G$ and $R_{\pp_G} =O_{G}/\pp_G$ is a finite extension of $C$.  The degree $[R_{\pp_G}:C]$ is called the degree of $\pp_G$.  If $\pp_H$ is a prime above $\pp_G$ in a finite extension $H$ of $G$, then $R_{\pp_H}$ is a finite extension of $R_{\pp_G}$ and
 \[%
 f(\pp_H/\pp_G)=[R_{\pp_H}:R_{\pp_G}]
 \]%
 is called the relative degree of $\pp_H$ over $\pp_G$ or over $G$.  The degree and the relative degree for infinite primes is defined analogously.
\end{definition}

Our first lemma deals with the issue of ramification for a function field extension.
\begin{lemma}%
\label{le:notramified}
If $H/G$ is a separable function field extension with $H=G(\alpha)$, for some $\alpha \in H$ integral with respect to a prime $\pp_G$ of $G$ (i. e. for every $H$-prime $\pp_H$ lying above a $G$-prime $\pp_G$ in $H$ we have that $\ord_{\pp_H}\alpha \geq 0$) and the discriminant of the power basis of $\alpha$ has order equal to zero at $\pp_G$, then $\pp_G$ has no ramified factors in $H$. (See \cite{C}, Chapter III, Section 8, Theorem 7 and Lemma 2.)
\end{lemma}%

Below is another property of ramification we need for our construction which can be deduced from the definition of ramification as a power in factorization.
\begin{lemma}
\label{le:prodram}
Let $G \subset H \subset L$ be a tower of function field extensions.  If $\pp_G$ is a prime of $G$ and $\pp_H$ and $\pp_L$ are primes above $\pp_G$ in $H$ and $L$ respectively, then $e(\pp_L/\pp_G) = e(\pp_L/\pp_H)e(\pp_H/\pp_L)$.
\end{lemma}

The next lemma is also a well known property of function field extensions.
\begin{lemma}%
\label{le:ef}
Let $L/N$ be a finite function field extension.  Let $\pp_N$ be a prime of $N$ and
let \[\pp_{L,1},\ldots,\pp_{L,n}\] be all the factors of $\pp_N$ in $L$.  Let $e_i$ be the ramification
degree of $\pp_{L,i}$ over $\pp_N$ and let $f_i$ be the relative degree  of $\pp_{L,i}$ over
$\pp_N$.  Then
\begin{equation}
\label{eq:ef}
\sum_{i=1}^ne_if_i=n.
\end{equation}
 (See \cite{C}, Chapter IV, Section 1, Theorem 1.)
\end{lemma}%
From this lemma we derive two corollaries to be used in our construction in Section \ref{sec:construction}.
\begin{corollary}
\label{cor:ef}
Let $L/N$ be a finite function field extension such that $[L:N]\not \equiv 0 \mod q$, where $q$ is a rational prime number.  In this case if  $\pp_N$ is a prime of $N$, then for some prime $\pp_L$ of $L$ lying above $\pp_N$, we have that $e(\pp_L/\pp_N) \not \equiv 0 \mod q$.
\end{corollary}
\begin{proof}
Suppose the assertion of the lemma is not true.  Let \[\pp_{L,1},\ldots,\pp_{L,n}\] be all the factors of $\pp_N$ in $L$.  If for all $i$ we have that $e(\pp_{L,i}/\pp_N) \equiv 0 \mod q$, then from \eqref{eq:ef} we conclude that $q | [N:L]$ contradicting our assumptions.
\end{proof}
\begin{corollary}
\label{cor:ef2}
Let $L/N$ be a finite function field extension such that $[L:N]\not \equiv 0 \mod q$, where $q$ is a rational prime number.  Let $\pp_N$ be a prime of $N$.  If $x \in N$ is such that $\ord_{\pp_N}x \not \equiv 0 \mod q$ then for some $L$ -prime $\pp_L$ lying above $\pp_N$, we have that $\ord_{\pp_L} \not \equiv 0 \mod q$.
\end{corollary}
\begin{proof}
By Remark \ref{rem:multe} for any $L$-prime $\pp_L$ above $\pp_N$ we have that $\ord_{\pp_L}x = e(\pp_L/\pp_N)\ord_{\pp_N}x$.  Thus our conclusion follows from Corollary \ref{cor:ef}.
\end{proof}
Next we need an elementary lemma and two corollaries to establish a property of a class of field extensions.
\begin{lemma}
\label{le:qthroot}
Let $G$ be any field and let $q$ be a prime number different from the characteristic of $G$.  Let $W \in G \setminus G^q$, where $G^q$ is the set of all the $q$-th powers of $G$.  Let $\beta$ be a root of $X^q - W$ in some algebraic closure of $G$.  Then $[G(\beta):G]=q$.
\end{lemma}
\begin{proof}%
Clearly $[G(\beta):G] \leq q$.  Suppose $\beta$ is of degree $m < q$ over $G$.  Any  conjugate of $\beta$ over $G$ is of the form $\xi_q^i\beta$, where $\xi_q$ is a primitive $q$-th of unity.    If $\xi_q^{i_1}\beta, \dots, \xi_q^{i_{m-1}}\beta$ are all the conjugates of $\beta$ over $G$ not equal to $\beta$, then ${\mathbf N}_{G(\beta)/G}(\beta)=\xi^j_q\beta^m \in G$, where $j \in \Z_{\geq 0}$.  Since $(q,m)=1$, there exist $a, b \in \Z$ such that $aq+bm = 1$ and thus $W^a(\xi^j_q\beta^m)^b=  \xi^{jb}_q\beta \in G$ making $W$ a $q$-th power in $G$ in contradiction of our assumptions.
\end{proof}%
\begin{corollary}
\label{cor:totdegree}
Let $G$ be any field, let $\{q_1,\ldots, q_n\}$ be a set of distinct prime numbers such that for some $\{A_1,\ldots, A_n \} \subset G$ each of the polynomials $T^{q_i}-A_i$ has no roots in $G$.  In this case $$[G(\sqrt[q_1]{A_1},\ldots,\sqrt[q_n]{A_n}):G]=q_1\ldots q_n.$$
\end{corollary}
\begin{proof}
By Lemma \ref{le:qthroot}, we have that $[G(\sqrt[q_1]{A_1},\ldots,\sqrt[q_n]{A_n}):G] \leq q_1\ldots q_n$ since $$[G(\sqrt[q_1]{A_1},\ldots,\sqrt[q_i]{A_i},\sqrt[q_{i+1}]A_{i+1}):G(\sqrt[q_1]{A_1},\ldots,\sqrt[q_i]{A_i})]  \leq q_{i+1}.$$  At the same time since $G(\sqrt[q_i]{A_i}) \subseteq G(\sqrt[q_1]{A_1},\ldots,\sqrt[q_n]{A_n})$ we have that $$[G(\sqrt[q_1]{A_1},\ldots,\sqrt[q_n]{A_n}):G] \equiv 0 \mod q_1\ldots q_n,$$ and the assertion of the corollary follows.
\end{proof}
\begin{corollary}%
\label{le:ramification}%
Let $G$ be a function field over a field of constants $C$.  Let  $q$ be a  prime number not equal  to the characteristic of $G$,  and let $W \in G \setminus G^q,$ where $G^q$ is the set of all the $q$-th powers in $G$.  If $\beta$ is a solution to $X^q -W=0$ in some algebraic closure of $G$, then the only primes of $G$ which ramify in the extension $G(\beta)/G$ are the primes which are zeros or poles of $W$ of order not divisible by $q$.  For the latter primes the ramification degree is $q$.
\end{corollary}%
\begin{proof}%
By Lemma \ref{le:qthroot}, we know that $[G(\beta):G]=q$.  Further,  note that the discriminant of the power basis of $\beta$, denoted by $D(\beta)$, is equal to $\prod_{0\leq i\not =j< m} (\beta_i -\beta_j)$, where $\beta_i = \xi_q^i\beta$, and $\xi_q$ is a primitive $q$-th root of unity.  Thus $D(\beta)=cW^{q-1}$ with $c \in C$, and if $\pp_G$ is a prime of $G$ and $\ord_{\pp_G}W=0$, then the conclusion follows immediately by Remark \ref{rem:const} and Lemma \ref{le:notramified}.  If, alternatively, we have that $\ord_{\pp_G}W \equiv 0 \mod q$, then by the Weak Approximation Theorem we can find an element $V \in G$ such that $-q\ord_{\pp_G}V=\ord_{\pp_G}W$ and replace $W$ by $W_1=WV^q$.  If we now let $\beta_1$ be a root of $X^q-W_1$, then clearly $G(\beta_1)=G(\beta)$ and $\ord_{\pp_G}D(\beta_1)=\ord_{\pp_G}W_1^{q-1}=0$.  Thus the the first assertion of the corollary holds.

Suppose now that $\ord_{\pp_G}W\not \equiv 0 \mod q$.    If $\pp_{G(\beta)}$ is a $G(\beta)$-prime above $\pp_G$, then by Remark \ref{rem:multe} we have that
\[
q\ord_{\pp_{G(\beta)}}\beta=\ord_{\pp_{G(\beta)}}W=e(\pp_{G(\beta)}/\pp_G)\ord_{\pp_G}W.
\]
Thus, $q | e(\pp_{G(\beta)}/\pp_G)$.    Since by Lemma \ref{le:ef} we have that $1 \leq e(\pp_{G(\beta)}/\pp_G) \leq [G(\beta):G]=q$, we must conclude that $e(\pp_{G(\beta)}/\pp_G)=q$.
\end{proof}%
Our next task is prove a series of technical propositions concerning evaluating the order at a prime.
\begin{lemma}%
\label{le:order}%
 Let $G$ be a function field over a field of constants $C$. Let $R(T) \in C(T)$. Let $s \in G \setminus C$ and let $\pp$ be a prime of $G$. Let $\displaystyle R(T)=\frac{A(T)}{B(T)}$, where $A(T), B(T) \in C[T]$ and are relatively prime in $C[T]$. If we suppose further that $A(T) = \prod_{i=1}^m(T-c_i)^{n_i}$ and $B(T) = \prod_{i=1}^k(T-b_i)^{j_i}$ with $c_i,b_i \in C$ and are all distinct, while all $n_i, j_i$ are positive integers, then the following statements are true.
\begin{enumerate}
\item  If $\pp$ is not a pole of $s$ and $\ord_{\pp}R(s) \not = 0$, we have that $\ord_{\pp}R(s) =n(c)\ord_{\pp}(s-c)$, where $c$ is the unique element of ${\tt R}=\{a_1,\ldots,a_m, b_1,\ldots, b_k\}$ with $\ord_{\pp}(s-c) >0$,  and $n(c)=n_i$, if $c=c_i$, and $n(c)=-j_i$, if $c=b_i$.
\item  If $\pp$ is not a pole of $s$ and $\ord_{\pp}R(s)=0$, then $\forall c \in {\tt R}: \ord_{\pp}(s-c) =0$.
\item  If $\pp$ is a pole of $s$, then $\ord_{\pp}R(s)=\ord_{\pp}s(\deg A(T) -\deg B(T))$.
\end{enumerate}
\end{lemma}%
\begin{proof}%
If $\pp$ is not a pole of $s$, then $\ord_{\pp}(s-c) \geq \min(\ord_{\pp}s, \ord_{\pp}c) \geq 0$ for all  $c \in {\tt R}$.   Thus in this case we have two possibilities: either
\begin{equation}
\label{ineq:1}
\exists c \in {\tt R}: \ord_{\pp}(s-c) >0
\end{equation}
or
\begin{equation}
\label{ineq:2}
\forall c \in {\tt R}: \ord_{\pp}(s-c) =0.
\end{equation}
In the former case we observe that the inequality in \eqref{ineq:1} can be true for at most one $c \in {\tt R}$.  Indeed, if we assume that for some  $c_1\not = c_2 \in \tt R$ we have that $\ord_{\pp}(s-c_1) >0$ and $\ord_{\pp}(s-c_2) >0$, then we will conclude that $\ord_{\pp}(c_1-c_2) >0$ which is impossible by Remark \ref{rem:const}.  Thus, if \eqref{ineq:1} holds,  we have that
\[
\ord_{\pp}R(s) =\sum_{i=1}^m \ord_{\pp}(s-a_i)^{n_i} -\sum_{i=1}^k\ord_{\pp}(s-b_i)^{j_i}= n(c)\ord_{\pp}(s-c).
\]

Now if \eqref{ineq:2} holds, then, clearly,
\[
\ord_{\pp}R(s)=0.
\]
Conversely, if $\ord_{\pp}R(s)=0$, then \eqref{ineq:2} must also hold in view of the argument above.  Thus the assertion of the lemma is true in this case also.

Finally, if $\ord_{\pp}s<0$, then
\[
\ord_{\pp}R(s) =\sum_{i=1}^m \ord_{\pp}(s-a_i)^{n_i} -\sum_{i=1}^k\ord_{\pp}(s-b_i)^{j_i}= \ord_{\pp}s(\sum_{i=1}^mn_i-\sum_{i=1}^kj_i) = \ord_{\pp}s(\deg A(T) -\deg B(T)).
\]
\end{proof}%
In the notation above, the following lemma gives necessary and sufficient condition for the order of $R(T)$ at some $G$-prime $\pp$ to be divisible by  $q$.
\begin{lemma}%
\label{le:notq}
Let $\displaystyle G, R(T)=\frac{A(T)}{B(t)},  \tt R$ be as in Lemma \ref{le:order}, let $q$ be a prime number, let $w \in G \setminus C, u \in C$.   If $\deg A(T) -\deg B(T) \not \equiv 0 \mod q$ and $\pp$ is a prime of $G$, then
\begin{enumerate}
  \item $\ord_{\pp}R(w-u) \equiv 0 \mod q$ if and only if either
  \begin{equation}
  \label{cond:1}
 \exists c \in {\tt R}: [n(c)\ord_{\pp}(w-u-c) \equiv 0 \mod q \land \ord_{\pp}(w-u-c) \not =0],
  \end{equation}
   or
   \begin{equation}
   \label{cond:2}
   \forall c \in {\tt R}: \ord_{\pp}(w-u-c) =0,
   \end{equation}
   or
   \begin{equation}
   \label{cond:3}
   \ord_{\pp}w <0 \land \exists c \in {\tt R}: [\ord_{\pp}(w-u-c) \equiv 0 \mod q ].
   \end{equation}
  \item $\ord_{\pp}R(w-u) \not \equiv 0 \mod q$ if and only if
  \begin{equation}
  \label{cond:4}
\ord_{\pp}w \geq 0 \land  \exists c \in {\tt R}: n(c)\ord_{\pp}(w-u-c) \not \equiv 0 \mod q,
\end{equation}
or
\begin{equation}
  \label{cond:5}
\ord_{\pp}w < 0 \land  \exists c \in {\tt R}: \ord_{\pp}(w-u-c) \not \equiv 0 \mod q.
\end{equation}
\end{enumerate}
 \end{lemma}%
 \begin{proof}%
 \begin{enumerate}
 \item Set $s=w-u$.  If $\ord_{\pp} s \geq 0$ (and thus $\ord_{\pp} w\geq 0$) and  $\ord_{\pp}R(s)=0$, then by Lemma \ref{le:order} we have that $\forall c \in {\tt R}: \ord_{\pp}(s-c) =0 \equiv 0 \mod q$.  If, alternatively,  $\ord_{\pp} s \geq 0$ (and thus again $\ord_{\pp} w\geq 0$),  while $\ord_{\pp}R(s)\not =0$,  and $\ord_{\pp}R(s) \equiv 0 \mod q$, then by Lemma \ref{le:order} again we have that
\[
\exists c \in {\tt R}: n(c)\ord_{\pp}(s-c) =\ord_{\pp}R(s)\equiv 0 \mod q.
\]
Finally if $\ord_{\pp}s <0$ and $\ord_{\pp}R(s) \equiv 0 \mod q$, then by Lemma \ref{le:order} once more, we have that
\[
(\deg A - \deg B)\ord_{\pp}s =\ord_{\pp}R(s) \equiv 0 \mod q.
\]
Given our assumption that  $\deg A(T) -\deg B(T) \not \equiv 0 \mod q$, we conclude that $\ord_{\pp}s \equiv 0 \mod q$ and thus for any $c \in \tt R$ we have $n(c)\ord_{\pp}(s-c) =n(c) \ord_{\pp}s \equiv 0 \mod q$.

Conversely, if   \eqref{cond:1} holds, then by Lemma \ref{le:order} such a $c$ is unique, and
\[
\ord_{\pp}R(s)=n(c)\ord_{\pp}(s-c) \equiv 0 \mod q.
\]
Further, clearly if Condition \eqref{cond:2} holds, we have that  $\ord_{\pp}R(s)=0$.  Finally, if $\ord_{\pp}s < 0$ and $\exists c \in {\tt R}: [\ord_{\pp}(w-u-c) \equiv 0 \mod q ]$, then $\ord_{\pp}s \equiv 0 \mod q$ and therefore $\ord_{\pp}R(s) \equiv 0 \mod q$.

\item First of all observe that if $\forall c \in {\tt R}: [n(c)\ord_{\pp}(s-c) \equiv 0 \mod q]$, then certainly $$\ord_{\pp}R(s) \equiv  0 \mod q.$$  At the same time,   if \eqref{cond:4} holds, then for the specified $c \in \tt R$ we have that $\ord_{\pp}(s-c) >0$ and by Lemma \ref{le:order}, $$\ord_{\pp}R(s) = n(c) \ord_{\pp}(s-c) \not \equiv 0 \mod q.$$ Finally, if \eqref{cond:5} holds, then $$\ord_{\pp}s=\ord_{\pp}(s-c) \not \equiv 0 \mod q$$ and thus $$\ord_{\pp}R(s) =(\deg A - \deg B)\ord_{\pp}s \not \equiv 0 \mod q.$$
\end{enumerate}
 \end{proof}%

\begin{lemma}%
\label{le:canfind}%
 Let $G, C, R(T), A(T), B(T), q$ be as in Lemma \ref{le:notq}.  Assume also that $\deg A > \deg B$.     Let $a \in G \setminus C$.  If $\calT$ is  a finite set
of primes of $G$ containing a pole $\ttt$ of $a$, then there exists $b \in G \setminus C$ such that  $\ord_{\ttt}R(b) \not \equiv 0 \mod q$, while for all $\pp \in \calT$ we have that $$\ord_{\pp}(R(a)^q -R(b)) \equiv 0 \mod q$$ and $$\ord_{\pp}(R(a)^q -R^{-1}(b)) \equiv 0 \mod q.$$
\end{lemma}%
\begin{proof}%
 By the Weak Approximation Theorem we can find $b \in G \setminus C$ such that $\ord_{\ttt}b=-1$,  and for all other $\qq \in \calT$ we have that $$|\ord_{\qq}b| > q|\ord_{\pp}R(a))|,$$ $$\ord_{\qq}b <0,$$ and  $$\ord_{\qq}b \equiv 0 \mod q.$$   Thus, given our assumptions on degrees of $A(T)$ and $B(T)$, by Lemma \ref{le:order} we have that $$\ord_{\ttt}R(b)=(\deg A - \deg B)\ord_{\ttt}b \not \equiv 0 \mod q$$ and $$\ord_{\ttt}R(b) <0.$$  At the same time observe that $$\ord_{\ttt}(R(a))^q =q(\deg A - \deg B)\ord_{\ttt}a  \equiv 0 \mod q$$ and $$\ord_{\ttt}(R(a)^q) < \ord_{\ttt}R(b) <0.$$   Therefore, $$\ord_{\ttt}(R(a)^q- R(b)) = \ord_{\ttt}R(a)^q \equiv 0 \mod q.$$ Further, for any $\qq \in \calT \setminus \{\ttt\}$, by Lemma \ref{le:order} again, we have that $\ord_{\qq}R(b) < q\ord_{\qq}R(a)$ and $\ord_{\qq}R(b) \equiv 0 \mod q$.  Consequently,   $$\ord_{\qq}(R(a)^q- R(b)) = \ord_{\qq}R(b) \equiv 0 \mod q.$$  At the same time,  for all $\pp \in \calT$ we have that $$\ord_{\pp}R^{-1}(b) > q\ord_{\pp}R(a),$$  and thus $$\ord_{\pp}(R(a)^q-R^{-1}(b)) = \ord_{\pp}R(a)^q  \equiv 0 \mod q.$$
\end{proof}%

\begin{lemma}%
\label{le:forevery}
 Let $G, C$ be as in Lemma \ref{le:order}.  Let $q$ be a rational prime different from the characteristic of $G$.   Let $w \in G$. Let $v
\in C$ be such that both $v- w$ and   $v- w^{-1}$ are $q$-th powers in $G$.  Then for any prime $\pp$ of $G$ we have that $\ord_{\pp}w
\equiv 0 \mod q$.
\end{lemma}%
\begin{proof}%
If $w \in C$ then    for any prime $\pp$ of $G$ we have that $\ord_{\pp}w=\ord_{\pp}w^{-1}=0\equiv 0 \mod q$.  If $w \in G\setminus C$ then   for any
prime $\pp$ of $G$ which is a pole of $w$  we have that $$\ord_{\pp}(v+w)=\ord_{\pp}w \equiv 0 \mod q.$$    Similarly     for any
prime $\pp$ of $G$ which is a zero of $w$  we have that $$\ord_{\pp}(v+w^{-1})=\ord_{\pp}w^{-1} \equiv 0 \mod q.$$
\end{proof}%

\section{Definability Construction}%
\label{sec:construction}
\setcounter{equation}{0}%
In this section we carry out the construction of our field.  We construct a field whose transcendence degree is one over the given field but the construction is easily extended to any positive transcendence degre.  Without loss of generality we assume that we have countably many sets to define.
\begin{notationassumption}%
\label{not:1}%
Below we will use the following notation and assumptions. %
\begin{itemize}%
\item Let $F$ be a countable field. Let $\{A_u, u \in \Z_{>0}\}$ be a sequence of countable subsets of $F$.%

\item Let $M$ be a countable field of transcendence degree at least one over $F$.  Let $t \in M$ be transcendental over $F$.%
\item For each $u\in \Z_{>0}$, let $R_u(S) \in F(S)$. Assume further that all the zeros and poles of $R_u(S)$ are in $F$ and at least one zero is
of order not divisible by $q_u$. This implies of course that at least one pole is of order not divisible by $q_u$. Without loss of generality we can
assume $$R_u(S) = \frac{C_u(S)}{B_u(S)},$$ where $C_u(S), B_u(S) \in F[S]$ and are relatively prime in $F[S]$. Assume also $$C_u(S) =
\prod_{i=1}^{m_u}(S-c_{u,i})^{n_{u,i}}$$ and $$B_u(S) = \prod_{i=1}^{k_u}(S-b_{u,i})^{j_{u,i}}$$ with $c_{u,i},b_{u,i} \in F$ and are all distinct, while
all $n_{u,i}, j_{u,i}$ are positive integers. Assume additionally $n_{u,1} \not \equiv 0 \mod q_u$ and $$0< \deg C_u(S) - \deg B_u(S) \not \equiv 0 \mod
q_u.$$  Let $${\tt R_u} = \{c_{u,1},\ldots,c_{u,m_u}, b_{u,1}, \ldots, b_{u,k_u}\}.$$  For $c \in {\tt R}_u$ let $n(c) = n_{u,i}$ if $c =a_{u,i}$, and  let $n(c) = j_{u,i}$ if $c=b_{u,i}$.
\item Let $\{q_u, u \in \Z_{>0}\}$ be a sequence of distinct prime numbers not equal to the characteristic of $M$. Assume also that $M$ contains
$q_u$-th roots of all its elements of the form $R_u(x)$ for any $x \in M$, for all $u \in \Z_{>0}$.%
\item Let $a_u =c_{u,1}$.  In other words, $a_u$ is a zero of  $R_u(T)$ with a multiplicity not divisible by $q_u$.
\item For all $u \in \Z_{>0}$ for any finite set $B$,  the complement of the set
\begin{equation}
\label{it:one}
\{x \in F: x =a+b, \mbox{ where } a \in A_u, b \in B \}
\end{equation} in $F$ is infinite.
\item Let $f_u(W,S)=W^{q_u}-R_u(S), u \in \Z_{>0}$.
\item Let $x \in M$.  Then by $\sqrt[q_u]{x}, u \in \Z_{>0}$ we will mean an element $y \in M$ such that $y^{q_u}=x$.
\item For any field $H$ such that $F\subset H$ and any $u \in \Z_{>0}$, let
\[%
A_{u,H}=\{z \in H: \exists t \in H \mbox{ such that } f_u(t,z)=0\}=
\]%
\[%
\{z \in H: \exists g \in H \mbox{ such that } g^{q_u}=R_u(z)\}=\{z \in H: R_u(z) \in H^{q_u}\}
\]%
\end{itemize}%
\end{notationassumption}

\begin{remark}
\label{rem:assumptions}
Observe that if Condition \eqref{it:one} does not hold, then  $A_u$ can be trivially defined over $F$, and since we construct a definition of $F$, we will cover these cases also.

\end{remark}

The main result of this and the following section is the theorem below.

\begin{theorem}%
\label{thm:main}
There exists a field $K \subset M$  of infinite
algebraic degree over $F(t)$ such that each $A_u, u \in \Z_{>0}$ is (first-order) definable over $K$.
\end{theorem}%
The proof is contained in the construction below.
\begin{construction}
We construct  $K$ to be an algebraic extension of $F(t)$ such that the following conditions are satisfied.
\be%
\item $F$ is definable over $K$ by a  first-order formula, in particular
\begin{equation}
\label{def:F}
F=\{a \in K| \forall b \in M: ((R_1(a)^{q_1}+R_1(b)) \in K^{q_1} \land  (R_1(a)^{q_1} + R^{-1}_1(b)) \in K^{q_1} ) \Rightarrow b \in A_{1,K})\}.
\end{equation}
\item For each $u \in \Z_{>0}$ we have that $A_u$ is definable by the first order formula over $K$ and in particular
\begin{equation}
\label{def:A_U}
A_u = \{r \in F| \forall r_1, r_2 \in F:(r_1 \not = r_2 \land r_1 + r_2 = r \Rightarrow ((t^{q_u}-r_1+a_u) \in A_{u,K} \mbox{ or } (t^{q_u}-r_2+a_u) \in A_{u,K}))\}
\end{equation}
\ee%
We will arrange for $K= \bigcup_{i=0}^{\infty} E_i$, where $F(t)=E_0 \subset E_1 \subset \ldots \subset M$ and $E_{i+1}/E_i$ is a finite algebraic extension. We will also construct finite sets $S_{i,u}, u \in \Z_{>0}$ contained in $E_i \cap (M\setminus (A_{K,u}\cup F))$. We set  $\emptyset = S_{0,u} \subset S_{1,u} \ldots$ and make sure that $$K \setminus (A_{K,u} \cup F)= ( \bigcup_{i \in \Z_{>0}}S_{i,u}).$$ In other words, all non-constant $K$-elements outside $A_{K,u}$ are contained in the union of $S_{i,u}$'s.

To satisfy the conditions above, we will require the following to be true at each step of our construction:
\begin{equation}
\label{eq:odd}%
\begin{array}{c}
\mbox{  For every }i, u \mbox{ and every }s \in S_{i,u}, \mbox{  there exists a prime } \pp_{s,i,u} \mbox{
of }E_i \\ \mbox{ such that } \ord_{\pp_{s,i,u}}R_u(s)  \not \equiv 0 \mod q_u.%
\end{array}%
\end{equation}%

\begin{equation}%
\label{eq:even}%
\forall u \in \Z_{>0} \forall r_1, r_2 \in F, r_1r_2\not=0: \mbox{ if } r_1 + r_2 \in A_u, \mbox{ then }
 \end{equation}
 \[%
\forall  s \in S_{i,u}: [ \ord_{\pp_{s,i,u}} R_u(t^{q_u} -r_1+a_u) \equiv 0 \mod q_{u}] \mbox{  or } \forall  s \in S_{i,u}: [\ord_{\pp_{s,i,u}} R_u(t^{q_u}-r_2+a_u) \equiv 0 \mod q_u].
\]%

We proceed by induction.  Let $\{x_i\}$ be a sequence of elements of $M$ algebraic over $F(t)$ such that every element appears infinitely
often.  Note that for $E_0=F(t), S_{0,u} =\emptyset, u \in \Z_{>0}$, Conditions \eqref{eq:odd} and \eqref{eq:even} are vacuously satisfied.

Assume now $(E_i, S_{i,u}, u \in \Z_{>0})$ have been constructed already and consider four cases below.\\

\noindent
$i=4n$:
If no $q_u$ divides $[E_i(x_n):E_i]$ then set $(E_{i+1},S_{i+1,u}) = (E_i(x_i),S_{i,u}), u \in \Z_{>0}$.%

We show that Condition (\ref{eq:odd}) holds for $(E_{i+1}, S_{i+1,u}, u \in\Z_{>0})$. Since $S_{i+1,u} = S_{i,u}$, by
induction it is enough to show that every prime of $E_i$ will have at least one factor in $E_{i+1}$
with a ramification degree not divisible by any $q_u$.  This follows from Lemma \ref{cor:ef2}. Thus for every $u$ and every $s \in S_{i+1,u}$ we can set $\pp_{i+1,u, s}$ to be a factor of $\pp_{i,u,s}$ with a ramification degree not divisible by $q_u$.   Further, since $S_{i,u}= S_{i+1,u}$, Condition  (\ref{eq:even}) carries over automatically by Remark \ref{rem:multe}.\\

\noindent%
$i=4n+1$:%
\begin{itemize}%
\item[(a)] If $x_n \not \in E_i$ then set   $(E_{i+1},S_{i+1,u}, u \in \Z_{>0}) = (E_i,S_{i,u}, u \in \Z_{>0})$ with $\pp_{i+1,u,s} =\pp_{i,u,s}$.
\item[(b)]   If $x_n \in E_i$, then follow the steps below.
\be%
\item  For each $u =1,\ldots,n$ do the following: \\
if for some $s \in S_{i,u}$ we have that $\ord_{\pp_{s,i,u}}R_u(x_n) \not \equiv 0 \mod q_u,$ then set $$(E_{i+1,u},S_{i+1,u}) = (E_i,S_{i,u}\cup\{x_n\});$$
  if for all $s \in S_{i,u}$ we have that $\ord_{\pp_{i,u,s}} R(x_n) \equiv 0 \mod q_u,$ and $R_u(x_n) \not \in A_{E_i,u}$ (or in other words $R_u(x_n)$ is not a $q_u$-th power in $E_i$), then  set $$(E_{i+1,u},S_{i+1,u}) = (E_i(\sqrt[q_u]{R_u(x_n)}),S_{i,u});$$
  if $R_u(x_n)  \in A_{E_i,u}$ (or in other words $R_u(x_n)$ is a $q_u$-th power in $E_i$), then set $$(E_{i+1,u},S_{i+1,u}) = (E_i,S_{i,u}).$$
\item For all $u \in \Z_{>n}$ set $$(E_{i+1,u},S_{i+1,u}) = (E_i,S_{i,u}).$$
\item Set $E_{i+1} = \prod_{i=1}^nE_{i+1,u}$ (by the product of fields we mean the compositum of fields, i.e. the smallest subfield of $M$ containing all the fields in the product).
\item For all $u \in \Z_{>0}$ and all $s \in S_{i+1,u}$ set $\pp_{i+1,u,s}$ equal to any factor of $\pp_{i,u,s}$ in $E_{i+1}$  with ramification degree not divisible by $q_u$.
\ee
\end{itemize}
We claim that all the parts of this step can be executed and  Conditions (\ref{eq:even}) and (\ref{eq:odd}) hold after this step.  First of all note that $[E_{i+1} :E_i]=\prod_{E_{i+1,u} \not = E_i}q_u$ by Corollary \ref{cor:totdegree}, and therefore for all $u$  such that $E_{i+1,u}=E_i$, and for all $s \in S_{i,u}$, every $\pp_{i,u,s}$ has a factor in $E_{i+1}$ with ramification degree not divisible by $q_u$ by Corollary \ref{cor:ef}.  Now let $u \in \{1,\ldots, n\}$ be such that $E_{i+1,u}=E_i(\sqrt[q_u]{R_u(x_n)})$ and note that we can separate the extension $E_{i+1}/E_i$ into a tower of two extensions: $$E_i \subseteq X_u \subset E_{i+1},$$ where $X_u=\prod_{j \not =u,E_{i+1,j} \not=E_i} E_{i+1,j}$. Observe that by Corollary   \ref{cor:totdegree}  again, the degree of the first extension is  equal to  $$\prod_{j \not =u,E_{i+1,j} \not=E_i}q_j$$ and every prime   corresponding to elements of $S_{i,u}$ will have a factor in  $X_u$ of ramification degree not divisible by $q_u$.  Finally, by Lemma \ref{le:notramified}, no factor of $\pp_{i,s,u}$ in  $X_u$ will be ramified in $E_{i+1}$.  Thus, by Lemma \ref{le:prodram}, we know that $\pp_{i,s,u}$ will have a factor in $E_{i+1}$ with ramification degree not divisible by $q_u$.  Therefore Condition \eqref{eq:odd} will still be satisfied after this step.  Since for every $u \in \Z_{>0}$ and every $s \in S_{i+1,u}$, the prime $\pp_{i+1,u,s}$ lies above a prime  $\pp_{i,u,s} \in S_{i,u}$, we have that Condition \eqref{eq:even} is satisfied after this step by induction hypothesis and Lemma \ref{cor:ef2}.\\

\noindent $i=4n+2$:
\be%
\item $x_n \not \in E_i$ or $x_n \in F$.  In this case $(E_{i+1},S_{i+1,u}, u \in \Z_{>0}) = (E_i,S_{i,u},u \in \Z_{>0})$.
\item  $x_n \in E_i \setminus F$. Let $\pp$ be a pole of $x_n$. This pole exists by Remark \ref{rem:const}.
For each $1 \leq u \leq n$ do the following.
\begin{enumerate}
\item Check to see if Condition (\ref{eq:even}) is satisfied for the set of primes $$\calA=\{p_{s,i,u}: s\in S_{i,u}\} \cup \{\pp\}.$$ If Condition (\ref{eq:even}) is not
satisfied, then for some $r_1, r_2 \in F$ with $r_1+r_2 \in A_u$ and $r_1r_2 \not = 0$ we have that for some $\qq \in \calA$, it is the case that $\ord_{\qq}(R_u(t^{q_u}-r_1+a_u)) \not \equiv 0 \mod q_u$  and for some $\ttt \in \calA$, it is the case that $\ord_{\ttt} (R(t^{q_u}-r_2+a_u)) \not \equiv 0 \mod q_u$. Since Condition (\ref{eq:even}) was previously satisfied for $\{\pp_{s,i,u}, s \in S_i\}$, we conclude that for some $r \not = 0$ we have that $\ord_{\pp}R(t^{q_u}-r+a_u) \not \equiv 0 \mod q_u$ and for all $\pp_{s,i,u}, s \in S_{i,u}$, $$\ord_{\pp_{s,i,u}}R(t^{q_u}-r+a_u) \equiv 0 \mod q_u.$$ We note that there can be only finitely many  $r \in F$ with such a property. Indeed, $\pp$ cannot be a pole of $t$ as otherwise $\ord_{\pp}(R(t^{q_u}-r+a_u)) \equiv 0 \mod q_u$ by Lemma \ref{le:notq}.  Thus, $$\ord_{\pp}(t^{q_u} - r-c+a_u) >0$$ for some $c \in {\tt R}_u$ by Lemma \ref{le:notq} again.   Similarly if the above condition held for some $\bar r \in F$ we would have  $\ord_{\pp}(t^{q_u} - \bar r-\bar c+a_u) >0$ for some $\bar c \in {\tt R}_u$.    Consequently,  $\ord_{\pp}(r- \bar r + c - \bar c)>0$.  The last inequality can hold only if  $r-\bar r +c - \bar c=0$.  However, $c - \bar c$ can take finitely many values only and therefore the set   $$V_{i,u, \pp} = \{r \in F: \ord_{\pp}R_u(t^{q_u}-r+a_u) \not \equiv 0 \mod q_u\land \forall s \in S_{i,u}, \ord_{\pp_{i,u,s}}R_u(t^{q_u}-r+a_u) \equiv 0 \mod q_u\}$$ is finite.

 Let $$E_{i,u}=E_i(\sqrt[q_u]{R(t^q-r+a_u)}, r \in V_{i,u,\pp}),$$ and observe that in the extension $E_{i,u}/E_i$ the prime $\pp$ will be ramified by Corollary \ref{le:ramification} while $$[E_{i,u}:E_i]=q_u^b, b \geq 1$$ by Lemma \ref{le:qthroot}.  Let $\bar \pp$ be a prime above $\pp$ in $E_{i,u}$.   At the same time, observe that  no $\pp_{i,u,s}, s \in S_{i,u}$ will be ramified in $E_{i,u}$ by Lemma \ref{le:ramification}.  Let $\bar \pp_{i,u,s}$ be an $E_{i,u}$-prime above $\pp_{i,u,s}$.
\item  By Lemma  \ref{le:canfind} there exists $b_u \in E_{i,u}$ such that
\be%
\item $ \ord_{\bar \pp}(R_{u}(x_n)^{q_u} + R_u(b_u)), \ord_{\bar\pp_{s,i,u}}(R_u(x_n)^{q_u}+ R_u(b_u)), s \in S_{i,u}$ are all divisible by $q_u$, and%
\item $\exists \qq \in \{\bar \pp, \bar \pp_{i,s,u}, s \in S_{i,u}\}$ such that $\ord_{\qq} R_u(b_u) =n_{u,1} \not \equiv 0 \mod q_u $.
\ee%
\item Define $(E_{i+1,u},S_{i+1,u})=(E_{i,u}( \sqrt[q_u]{R_u(x_n^{q_u})+R_u(b_u)}), S_{i,u}\cup \{b_u\}).$  Observe that $$[E_{i+1,u}:E_{i,u}]=q_u^b, b=0,1.$$  Further, no prime from the set $\{\bar \pp, \bar \pp_{i,s,u}, s \in S_{i,u}\}$ is ramified in this extension by Lemma \ref{le:ramification}.  Let $\hat \pp$ lie above $\bar \pp$ in $E_{i+1,u}$ and for every $u=1,\ldots,n$ and every $s \in S_{i,u}$, let $\hat \pp_{i,s,u}$ lie above $\bar \pp_{i,s,u}$ in $E_{i+1,u}$
\ee
\item Set $E_{i+1}=\prod_{u=1}^{n}E_{i+1,u}.$  Note that all the prime factors of $[E_{i+1}:E_i]$ are in the set $\{q_1,\ldots,q_n\}$ by Corollary \ref{cor:totdegree} and no prime from the set $\{\hat \pp, \hat \pp_{i,s,u}, s \in S_{i,u}\}$ is ramified in the extension $E_{i+1}/E_{i+1,u}$ for any $u=1,\ldots, n$ by Lemma \ref{le:ramification} again.  Therefore for every $u =1,\ldots,n$ we can let $\pp_{i+1,u,s}$  be any factor of $\hat \pp_{i,u,s}$ in $E_{i+1}$.  Observe that for all $u=1,\ldots,n$ and all $s \in S_{i,u}$ we have that $$e(\pp_{i+1,u,s}/\pp_{i,u,s})=e(\pp_{i+1,u,s}/\hat \pp_{i,u,s})e(\hat \pp_{i,u,s}/\bar \pp_{i,u,s})=e(\bar \pp_{i,u,s}/\hat\pp_{i,u,s})=1$$ and therefore for any $u=1,\ldots,n$ and any $s \in S_{i,u}$ we still have that $$\ord_{\pp_{i+1,s,u}}s \not \equiv 0 \mod q_u.$$ Next let $\pp_{i+1,u,b_u}$ be any factor of $\hat \pp$ in $E_{i+1}$ and observe that $$e(\pp_{i+1,u,b_u}/\bar \pp) =e(\pp_{i+1,u,b_u}/\hat \pp)e(\hat \pp/\bar \pp)=1$$ and therefore $\ord_{\pp_{i+1,u,b_u}}b_u=\ord_{\bar \pp}b_u \not \equiv 0 \mod q$.

    For every $u \in \Z_{>n}$ set $S_{i+1,u}=S_{i,u}$.  For any $s \in S_{i,u}$ let $\pp_{i+1,u,s}$  be any factor $\pp_{i,u,s}$ in $E_{i+1}$ with ramification degree of $E_i$ not divisible by $q_u$.  (As before such a factor exists by Corollary \ref{cor:ef2}.)  Given this choice we now have for all $u >n$ and all $s \in S_{i,u}=S_{i+1,u}$ that $\ord_{\pp_{i+1,u,s}}s \not \equiv 0 \mod q_u$ by Remark \ref{rem:multe}
\ee%

We claim that  Conditions \eqref{eq:even} and \eqref{eq:odd} are satisfied after this step.  If $x_n \not \in E_i$ or $x_n \in F$, there is nothing to prove.    So assume that we are in the case of $x_n \in E_i \setminus F$.  From the discussion above it is clear that Condition \eqref{eq:odd} is satisfied.   Next we note that $E_{i,u}$ was constructed explicitly so that Condition \eqref{eq:even} held in $E_{i,u}$ for all primes in $\calA$.   Note that each $\pp_{i+1,u,s}, s \in S_{i+1,u}$ is a factor of a prime in $\calA$.  Thus, by Remark \ref{rem:multe}, Condition \eqref{eq:even} is still satisfied in $E_{i+1}$ for $u\leq n$.\\  At the same time for $u>n$ Condition \ref{eq:even} is satisfied by induction  and Remark \ref{rem:multe} again, since $S_{i,u}=S_{i+1,u}$.

$i=4n+3$:
For each $u=1,\ldots,n$ do the following.
\begin{enumerate}
\item If $x_n \in A_u$ or if $x_n \not \in F$, then $(E_{i+1,u}, S_{i+1,u})=(E_i,S_i)$.
\item  Assume $x_n \in F \setminus A_u$.  Let  $\bar \pp_r$ be a prime of $F(t)$ corresponding to $t-\sqrt[q_u]{r}$ if $r$ is a $q_u$-th power in $F$ or to $t^{q_u}-r$ if $r$ is not a $q_u$-th power in $F$, and let $\pp_r$ be any $E_i$-prime lying above $\bar \pp_r$.  Note that for all $r \not = 0$ we have that  $\ord_{\bar \pp_r}(t^{q_u}-r)=1$ and for all but finitely many $r \in F$ we also have that $\ord_{\pp_r}(t^{q_u}-r)=1$.  (It will be bigger than 1 only if $\bar \pp_r$ ramifies in the extension $E_i/F(t)$ and there are only finitely many ramified primes.) If  $c \in {\tt R}_u$ then
    \begin{equation}
    \label{eq:depends}
   \ord_{\pp_r}(t^{q_u}-r+a_u - c) = \left \{ \begin{array}{c} e(\pp_r/\bar \pp_r)n_{1,u}  \mbox{ if  } c = a_u = c_{1,u} \\ \ord_{\pp_r}(a_u - c) = 0 \mbox{ if } c \not = a_u =c_{1,u} \end{array} \right .
\end{equation}
Thus
\begin{align}
\ord_{\pp_r}R_u(t^{q_u}-r+a_u) =\\
 \left ( \sum_{z=1}^{k_u}n_{z,u}\ord_{\pp_r}(t^{q_u}-r+a_u-c_{z,u}) \right ) - \left ( \sum_{z=1}^{m_u}j_{z,u}\ord_{\pp_r}(t^{q_u}-r+a_u-b_{z,u}) \right )= \\
 e(\pp_r/\bar \pp_r)n_{1,u} \not \equiv 0 \mod q_u
\end{align}
for all but finitely many primes, and the  set $$C_u=\{r \in F^*| \ord_{\pp_r}R_u(t^{q_u}-r+a_u) \equiv 0 \mod q_u\}$$ is a finite set.  Next let $B_u$ be the union of the two sets below:
\[
 \{ g \in F:  \exists s \in S_{i,u} \ord_{\pp_{i,u,s}} R_u(t^{q_u}-g+a_u)  \not \equiv 0 \mod q\}
 \]
 \[
 \{g \in F: \ord_{\pp_{r_j}}R_u(t^{q_u}-g+a_u) \not \equiv 0 \mod q, j=1,2\},
\]
 and observe that $B_u$ is a finite set.  Indeed, for any prime $\pp$ of $E_i$ the set of $g \in F$ such  that $\ord_{\pp}R_u(t^{q_u}-g+a_u) \not \equiv 0 \mod q$ is finite.  By Lemma \ref{le:order}, $\pp$ cannot be a pole of $t$ since in that case $\ord_{\pp}R_u(t^{q_u}-g+a_u) \equiv 0 \mod q$.  Therefore by the same lemma  we have that $\ord_{\pp}(t^{q_u}-g+a_u -c) >0$ for exactly one $c \in \tt R_u$.  But for some unique  $a_{\pp} \in F$ we have that $\ord_{\pp}(t-a) >0$ and therefore we conclude that $a_{\pp}^{q_u}-g +a_u-c=0$, giving us a unique value of $g$.

Let $r_1 \in F$ be such that  the following conditions are satisfied:
\begin{eqnarray}
\label{eq:r1} r_1 \not \in C_u,\\
\label{eq:r2} r_1 \not = x_n -x, \mbox{ where } x \in C_u,\\
\label{eq:leftover} r_1 \not \in A_u -a_u + \tt R_u - B_u.
\end{eqnarray}
Next let $r_2= x_n - r_1$.  By assumption (see Notation and Assumptions \ref{not:1}), infinitely many elements of $F$ satisfy  \eqref{eq:leftover}, and only finitely many elements of $F$ can fail to satisfy  \eqref{eq:r1} and \eqref{eq:r2}. Thus we can certainly choose $r_1$ to satisfy \eqref{eq:r1} -- \eqref{eq:leftover}.

Set $E_{i+1}=E_i$,  set $S_{i+1,u}=S_{i,u}\cup \{R_u(t^{q_u}-r_1+a_u), R(t^{q_u}-r_2+a_u)\}$ for $u =1,\ldots,n$  and $S_{i+1,u} = S_{i,u}$ for $u>n$. Since $r_1, r_2 \not \in C_u$, we have that $\ord_{\pp_{r_i}}R_u(t^{q_u}-r_i+a_u) \not \equiv 0 \mod q_u$ and therefore Condition \eqref{eq:odd} is satisfied.

Suppose now that Condition \eqref{eq:even} is not satisfied for some $u$ and for some $\bar r_1 + \bar r_2 \in A_u$.
Since  Condition \eqref{eq:even} was satisfied before this step we can assume without loss of generality that $\ord_{\pp_{i,u,s}}R_u(t^{q_u}- \bar r_1+a_u) \equiv 0 \mod q_u$ for all $s \in S_{i,u}$ and $$\ord_{\pp_{r_1}}R(t^{q_u}-\bar r_1+a_u) \not \equiv 0 \mod q_u.$$  Since we know $\pp_{r_1}$ not to be a pole of $t$, we know that $\ord_{\pp_{r_1}}(t^{q_u}-\bar r_1+a_u-c) >0$ for some $c \in {\tt R}_u$ by Lemma \ref{le:notq}.  But  we also have that    $\ord_{\pp_{r_1}}(t^{q_u}-r_1) >0$ and therefore $\ord_{\pp_{r_1}}(r_1-\bar r_1+a_u-c )>0$ implying that
\begin{equation}
\label{eq:rel1}
r_1=\bar r_1-a_u +c_1
\end{equation}
by Remark \ref{rem:const}.  Thus we have that $r_1+\bar r_2=a -a_u+c$, where $a \in A_u$ and $$\bar r_2 =a -a_u+c-r_1 \not \in B_u,$$ by assumption on $r_1$, implying that $$\ord_{\pp_{i,u,s}}R_u(t^{q_u}- \bar r_2+a_u) \equiv 0 \mod q_u$$ for all $s \in S_{i,u}$ and $$\ord_{\pp_{r_1}}R(t^{q_u}-\bar r_1+a_u) \equiv 0 \mod q_u.$$  Consequently, Condition \eqref{eq:even} is still satisfied after these steps.
\end{enumerate}%
\end{construction}
\begin{remark}%
Step $4n$ can be omitted without changing the definability properties of $K$.
\end{remark}%

\section{Properties of $K$.}
\setcounter{equation}{0}
In this section we complete the proof of Theorem \ref{thm:main} by proving that the constructed field $K$ has the desired properties.
\begin{theorem}
\label{thm:properties}
\be%
\item For all $u \in \Z_{>0}$ we have that $K \setminus A_{u,K}=\bigcup_{i \in \Z_>0}S_{i,u}$.
\item Formula \eqref{def:F} holds over $K$.
\item  Formula \eqref{def:A_U} holds over $K$.
\item If $H \subset M$ is such that $K \subseteq H$ and $[H:K]<\infty$, then $[H:K]=1$ or $[H:K] \equiv 0 \mod q_u$ for some $u$.
\ee
\end{theorem}

\begin{proof}%
\be
\item Let $S_u = \bigcup_{i \in \Z_{>0}}S_{i,u}$. Suppose $x \in S_u$ and $x \in A_{K,u}$. Then for some $i$ we have that $x $ is a
 in $A_{E_i,u}$ while $x \in S_{i,u}$. But this would contradict Condition \eqref{eq:odd} requiring that for some prime
$\pp_i$ of $E_i$ we have that $\ord_{\pp_i}R_u(x) \not \equiv 0 \mod q_u$.

Suppose $x \in K$ and $x \not \in A_{K,u}$.  Let $n$ be such that $x$ was added to $E_i$ for some $i <n$ and $x = x_n$.  In Step
$4n+1$ we check if for some $s \in S_{i,u}$ it is the case that $\ord_{\pp_{s,i,u}}R_u(x) \not \equiv 0 \mod q_u$.  if such an $s$ is found
$x$ is added to $S_{i,u}$ and thus $x \in S_u$.  Suppose no $s$ with $\ord_{\pp_{s,i,u}}x \not \equiv 0 \mod q_u$ was found.  In this case
we add $\sqrt[q_u]{R_u(x)}$ to $E_i$ and $x$ becomes an element of $A_{K,u}$ contradicting our assumption.  Thus, $x \in
S_u$.%

\item Let $x \in K$ and assume that in some $E_i$ containing $x$ the divisor of $R_1(x)$ is a $q_1$-th power of another divisor. Then $R_1(x)$
is a $q_1$-th power in $K$.   Indeed, let $n$ be such that $x \in E_i$ for some $i < n$ and $x_n =x$.
Then since the order of all zeros and primes of $R_1(x)$ is divisible by $q_1$,  in the step $4n+1$ we have that   $\sqrt[q_1]{R_1(x)}$ is added to $E_n$.\\

Suppose now that $a \not \in F$. Then by step $4n+2$ of the construction there exists $b \in K$ such that $R_1(b)$ is not a $q_1$-th power but $R_1(b)+R_1(a)^{q_1}$ is a $q_1$-th power.

Let $a \in F$, let $b \in K$ satisfying $R_1(a)^{q_1} +R_1(b) \in K^{q_1} \land R_1(a)^{q_1} +\frac{1}{R_1(b)} \in K^{q_1}$ Then in some $E_i$ we have that the
order at all the poles and zeros of $R_1(b)$ is divisible by $q_1$. Hence by the argument above, $R_1(b)$ is a $q_1$-th power in $K$.

\item Let $u$ be a positive integer.  Suppose $r \in F$ and  $r \not \in A_u$.  Then by Part $4n+3$ of the construction  $\exists r_1, r_2 \in F$ with  $r_1 r_2 \not = 0$ such that $(t^{q_u}-r_1+a_u)$ and $ (t^{q_u}-r_2+a_u)$  are in $S_u$ and  therefore cannot be in $A_{K,u}$ by Part 1 of this proof.  Suppose now that $r \in A_u$, $r_1, r_2 \in F, r_1r_2  \not = 0, r_1+r_2 = r$ but $(t^{q_u}-r_1+a_u \not \in A_{K,u})\land (t^{q_u} -r_2+a_u \not \in A_{K,u})$.  In this case $(t^{q_u}-r_1 + a_u\in S_{i,u}) \land(t^{q_u} -r_2 +a_u\in  S_{i,u})$ for sufficiently large $i$ by Part 1 of this proof.  But this contradicts  Condition \eqref{eq:even}.

\item Suppose there exist an element $\alpha \in M$ algebraic over $K$ and of degree $n$ such that $(n,q_u) = 1$ for all $u$.  If $a_0,\ldots,a_{n-1} \in K$ are the coefficients of  the  monic irreducible polynomial of $\alpha$ over $K$ and  $E_i$ is such that $D_0,\ldots,D_{n-1}$, then $\alpha$ is of degree not divisible by any $q_u$ over any $E_j$ with $j \geq i$.  In this case, however, $\alpha$ would have been added to some $E_j$ in a step $4n$ for some $n$.
\ee%
\end{proof}%

\section{Undecidability of Theories}%
\label{sec:undecidability}
\setcounter{equation}{0}%
In this section we prove Theorem \ref{introthm2} by specializing the Section \ref{sec:construction}.  Let $U, \calQ, \calR, \calP$ be defined as in the statement of this theorem. By Proposition \ref{introprop}  we can assume that there exists
\[
G \models \mathfrak T_{U, \calQ, \calR, \calP}
\]
of transcendence degree at least one over $U$,  if $U$ is not algebraic over a finite field and of transcendence degree at least 2 if $U$ is algebraic over a finite field.   For each $q$ implement the construction above with $M=G$, $F=U$, if $U$ is not algebraic over a finite field, and $F=U(t)$ with $t$ transcendental over $U$, if $U$ is algebraic over a finite field.  Further let $q_1=q$,  $A= \Z$ if the characteristic is 0, and $A=\F_p[t]$ if the characteristic is $p>0$.

 The constructed field $K=K_q$ will satisfy the following conditions:
\be
\item  $A$ is definable over $K_q$.
\item Any non-trivial finite extension of $K_q$ in $G$ is of degree divisible by $q$.
\ee
Now let $\Omega$ be any non-principal ultra-filter and let $L = \prod_{q}K_q/\Omega$ and let $\hat G = \prod_{q \in \calQ}G/\Omega$.  By a standard argument using {\L}o\u{s}'s Theorem, we have that $L$ is algebraically closed in $\hat G$, while $\hat G$ is a model of $\Tt_{U,\calQ,\calR,\calP}$.  Thus we have that $L$ is also a model of $\Tt_{U,\calQ,\calR,\calP}$ and consequently any finite subtheory of this theory has at least one $K_q$ as its model, making the subtheory undecidable.

Below are some examples of theories covered by our theorem.
\begin{example}
A theory $\Tt_{U, \calR, \calP}$, where $\calQ$ consists of all the primes equivalent to 3 mod 5,  $R_q(T)=T$,  $P_i(T)=T^{q_i} -3$, with $q_i$ a prime not in $\calQ$, and $U$ is the field obtained from $\Q$ by adjoining all the $q^n$-th  roots of all integers for all $q \in \calQ$ and all $n \in \Z_{>0}$.
\end{example}
\begin{example}
  A theory of a field of any characteristic not equal to a prime number $p$, where all the polynomials of degree not divisible by $p$ have a root in the field and some polynomials of degree $p$ do not have a root. Here we also let $R_q(T)=T$, for $q\not = p$.
\end{example}

\section{Some Open Questions}
One of the motivations of this paper was our interest in finding new finitely hereditarily undecidable theories, in particular in cases where the original theory is decidable.  In this connection we have the following questions.
\be
\item It is known that the theory of the field of all totally real algebraic numbers is decidable (see \cite{Fried3}).  Is this theory  finitely hereditarily undecidable?
\item Is the theory of pseudo real closed fields (PRC fields) finitely hereditarily undecidable?
\item In general, if ${\mathfrak T}$ is a theory of any subfield of $\bar \Q$, the algebraic closure of $\Q$, is ${\mathfrak T}$ finitely hereditarily undecidable?  (We know that this is true for number fields.)
\ee

\end{document}